\documentclass[12p]{amsart}

\newtheorem{lemma}{Lemma}
\newtheorem{theorem}{Theorem}
\newtheorem{proposition}{Proposition}

\usepackage{amssymb}
\usepackage{amsmath}

\begin{document}
\pagestyle{headings}  % switches on printing of running heads

\title[]{Operators that `coerce' \\ the surjectivity of convolution}

\date{15 December 2013}

\author{Richard F. Bonner}

\thanks{The author acknowledges a former scholarship of the Centre for Mathematical Analysis in Canberra, a partial support by the Swedish Institute, grant 01424/2007, and the recent reduction in teaching duties at his present place of accreditation.} 

\address{M\"{a}lardalen University College, V\"{a}ster\aa s, Sweden.}

\email{richard.bonner@mdh.se}

\maketitle

%\begin{abstract}
%
%
%\end{abstract}

\section{Introduction}

We relate to earlier \cite{hormander_63,abramczuk,bonner} and more recent work \cite{bonet-etal,fernandez-etal} about perturbation of surjective convolution.
Let $\Phi$ be a space of smooth functions in $R^n$ with dual $\Phi'$ consisting of (ultra-) distributions with compact support.
Following Ehrenpreis \cite{ehrenpreis_56,ehrenpreis_60} one calls $u \in \Phi'$ invertible if $u$ acts surjectively in $\Phi$ via convolution.
In addition to Ehrenpreis', there is considerable body of work about invertibility in various settings, especially H\"{o}rmander's \cite{hormander_62,hormander_63,hormander_68,hormander_79} for the ground case of Schwartz distributions, and numerous others' for more general classes, in particular \cite{bonet-etal,braun-etal,chou,frerick-wengenroth,cioranescu_80,cioranescu_zsido,cioranescu_zsido_arkiv}.

However, the general characterizations of invertibility, it is in the nature of things, are often ineffective in specific cases, and so the question about the stability of invertibility under perturbations naturally arises. 
We are therefore interested in linear maps $p$ in $\Phi'$ leaving the set of the non-invertible elements stable, thus `preserving' non-invertibility: if $u$ is non-invertible then so is $pu$.
One may then say (for lack of a better term) that $p$ coerces invertibility, or, for short, that $p$ is coercive: if $pu$ is invertible then so is $u$.
Obviously, should the inverse $p^{-1}$ exist, then $p$ preserves non-invertibility if and only if $p^{-1}$ preserves invertibility.

For the Schwartz distributions with compact support,
$\Phi'=\mathcal{E}'=\mathcal{E}'(R^n)$, 
it was shown in \cite{hormander_79}, and with a simpler proof in \cite{abramczuk}, that multiplication by a real analytic function is coercive.
It is then not hard to see that `convolution operators with real analytic coefficients', the finite sums of convolution by elements of $\mathcal{E}'$ composed with multiplication by analytic functions, are coercive.
The next step, by analogy with the passage from the differential to pseudo-differential case, is to consider `pseudo-convolution operators'; the real analytic so generalized operators, we show, indeed are coercive on $\mathcal{E}'$.

The multiplication by a non-analytic function $f$ will in general not be coercive, but it may coerce invertibility on a space of smoother functions, the regularity of $f$ determining how much smoother; we quantify this phenomenon with the scale of Beurling spaces $\mathcal{E}_w'$, while measuring the regularity of $f$ with the Beurling classes $\mathcal{E}_w$ and the Denjoy-Carleman classes $C^L$; the Gevrey case looks particularly neat.
Generalizing the multipliers are pseudo-convolution operators, considered on $\mathcal{E}_w$, regular in the $C^L$ sense.
It follows that real analytic such operators with parametrix of the same form, in particular the elliptic pseudo-differential operators, both coerce invertibility and preserve it.
%
%
%; a case, cf Theorem \ref{thm-pdo}, is an elliptic analytic pseudo-differential operator acting on $\mathcal{E}'$. 

\section{Functions and distributions} \label{spaces}

Recall the standard definitions, referring to \cite{bjorck_66,bjorck_71,chou,hormander_79,hormander_pde} for background.

\subsection{Denjoy-Carleman classes} \label{d-c}

Let $L= (L_0, L_1, \ldots )$, $L_0=1$, be an increasing sequence of positive numbers satisfying $C^{-1} L_{k+1} \geq L_k \geq k$ for $k \geq 0$ and some constant $C > 0$.
For $K \subset R^n$ compact, $f$ smooth (complex-valued) function in a neighbourhood of $K$, and $r > 0$, put
\begin{equation}\label{dc-norm}
|f|_{L,r,K} := \sup_{\alpha \geq 0, x \in K} ( \frac{r}{L_{|\alpha|}})^{|\alpha|} \cdot |D^\alpha f(x)|.
\end{equation}

One says that $f$ is of class $C^L$ on $K$ if $|f|_{L,r,K}$ is finite for some $r > 0$, writing $f \in C^L(K)$.
For $X \subset R^n$ open, let $C^L(X)$ be the set of all smooth functions in $X$, which are of class $C^L$ on every compact set in $X$.
Recall that $C^L(X)$ is called a quasi-analytic class if it contains no non-trivial element with compact support; the exact condition for this is the divergence of the integral $\int_1^{+ \infty} q_L (t) \cdot t^{-2} \, dt$, where $q_L(t)$ denotes the logarithm of the least upper bound of $(t/L_k)^k$, $k \geq 0$.

\subsection{Beurling classes} \label{beurling}

\subsubsection{Weight functions}

Let $M$ be the set of all non-negative sub-additive functions $w$ on $R^n$, normalized by $w(0)=0$, each bounded from below by a function of form $a + b\cdot \log (1+|\xi|)$ with $a$ real, $b$ positive, and $\xi \in R^n$, and with growth bounded at infinity in the integral sense
\begin{equation}\label{bound}
\int \frac{w(\xi)}{1+|\xi|^{n+1}} \, d\xi < \infty.
\end{equation}
Assume for simplicity the functions in $M$  symmetric, $w(\xi)=w(-\xi)$, $\xi \in R^n$.

For positive functions $w_1$, $w_2$, in $R^n$, write $w_2 \succ w_1$, and say that $w_2$ dominates $w_1$, if for some constants $A$ real and $B$ positive, and all $\xi \in R^n$,
\begin{equation}\label{domination}
w_2(\xi) \geq A + B \cdot w_1(\xi).
\end{equation}
Write $w_1 \sim w_2$ if both $w_2 \succ w_1$ and $w_1 \succ w_2$ hold.
Write $w_2 \succ \succ w_1$, and say that $w_2$ strictly dominates $w_1$, if for some real constant $A$, some function $B$ with $\lim_{|\xi| \rightarrow \infty} B(\xi) = \infty$, and all $\xi \in R^n$,
\begin{equation}\label{strict-domination}
  w_2(\xi) \geq A + B (\xi) \cdot w_1(\xi).
\end{equation}
In short, $w_2 \succ w_1$ if $w_2 = O(w_1)$, and $w_2 \succ \succ w_1$ if $w_2 = o(w_1)$, for $|\xi|$ large. 

Recall (cf. Theorem 1.2.7 in \cite{bjorck_66}) that every $w \in M$ is dominated by some $\tilde{w} \in M$ of the form $\tilde{w}(\xi)= \Omega(|\xi|)$ with $\Omega$ concave; and $\Omega$ may be chosen so that $\tilde{w} \succ \succ w$.

We say that $w \in M$ is \emph{slowly varying} if
\begin{equation}\label{slowly-varying}
\inf_{x \in B(\xi, \delta(\xi))} w(x) \ \ \succ \sup_{x \in B(\xi, \delta(\xi))} w(x)
\end{equation}
whenever $\delta$ is a positive function in $R^n$ and $\delta(\xi)=o(|\xi|)$ as $|\xi| \rightarrow \infty$; as usual, $B(x.r)$ stands for the Euclidean ball of radius $r$ at $x$.
Note that the functions in $M$ that are monotone in the radius are slowly varying.

\subsubsection{Spaces $\mathcal{D}_w$ and $\mathcal{E}_w$.}

For $\phi \in C_0^\infty(R^n)$, $w \in M$, and $\lambda > 0$, put
\begin{equation}\label{w-norm}
\|\phi\|_{\lambda}^w  := \int |\hat{\phi}(\xi)| \,  e^{\lambda w(\xi)} \, d\xi ,
\end{equation}
with $\hat{\phi}$ denoting the Fourier transform of $\phi$, and let $\mathcal{D}_w$ be the set of $\phi \in C_0^\infty(R^n)$ with $\|\phi\|_{\lambda}^w$ finite for every $\lambda > 0$.
Due to condition (\ref{bound}), every $\mathcal{D}_w$ contains `local units', that is, for every $w \in M$, whenever $K \subset X \subset R^n$, $K$ compact and $X$ open, there exists $\phi \in \mathcal{D}_w$ with support in $X$ and equal to one in a neighbourhood of $K$; see \cite{bjorck_66}.
Recall also that $\mathcal{D}_{w_1} \subset \mathcal{D}_{w_2}$ if and only if $w_1 \succ w_2$.
Topologize each $\mathcal{D}_{w}$ in the standard way, as in the Schwartz case
$w(\xi)= \log (1+|\xi|)$.

For $f \in C^\infty(R^n)$, $\phi \in \mathcal{D}_w$, and $\lambda > 0$, put
\begin{equation}\label{wl-norm}
\|f\|_{\lambda, \phi}^w := \|\phi f\|_{\lambda}^w  ,
\end{equation}
and write 
$\mathcal{E}_w$ for the set of $f \in C^\infty$ with $\|f\|_{\lambda, \phi}^w$ 
finite for all $\phi \in \mathcal{D}_w$ and $\lambda > 0$.
The set $\mathcal{E}_w$ is a Frechet space under the semi-norms (\ref{wl-norm}).

Recall that
\begin{equation}\label{w-norm-fourier}
|||\phi |||_\lambda^w := \sup_{\xi \in R^n} |\hat{\phi}(\xi)| \, e^{\lambda w(\xi)}
\end{equation}
with $\lambda > 0$ is an equivalent set of semi-norms on $\mathcal{D}_w$; then $\mathcal{E}_w$ is equivalently topologized using the semi-norms 
$$ ||| f|||_{\lambda, \phi}^w := |||\phi f |||_{\lambda}^w, \  \ \phi \in \mathcal{D}_w,   \lambda > 0.$$

When 
$w(\xi)=|\xi|^\alpha$, $0 < \alpha < 1$, 
the spaces $\mathcal{D}_w$ and $\mathcal{E}_w$ will be noted by 
$\mathcal{D}_{(\alpha)}$ and 
$\mathcal{E}_{(\alpha)}$,
respectively; these are known as the (`small') Gevrey spaces.

For $X \subset R^n$ open, one defines $\mathcal{D}_w(X)$ and $\mathcal{E}_w(X)$ as above, but admitting $\phi \in C_0^\infty(X)$ only.

\subsubsection{Spaces $\mathcal{D}_w'$, $\mathcal{E}_w'$, and $\mathcal{F}_w$.}

Recall, duality in Beurling spaces works essentially as in the Schwartz case.
The space $\mathcal{D}_w$ embeds topologically with dense image in $\mathcal{E}_w$, which gives an embedding of the strong duals, $\mathcal{E}_w'$ into $\mathcal{D}_w'$.
Due to the existence of partition of unity in $\mathcal{E}_w$, implied by condition (\ref{bound}), the dual $\mathcal{E}_w'$ is thus identified with distributions with compact support in $\mathcal{D}_w'$.
The notion of $w$-singular support of $u \in \mathcal{D}_w'$ is defined in the usual way.
One may then define convolutions; in particular, $\mathcal{E}_w'$ receives ring structure.
The Fourier-Laplace transform of $u \in \mathcal{E}_w'$, defined by 
$\hat{u}(\zeta):= \langle u, e^{-i \, \langle \cdot, \zeta \rangle} \rangle$, 
$\zeta = \xi + i \eta$, $i := \sqrt{-1}$, 
gives then by Paley-Wiener a topological isomorphism of the convolution ring $\mathcal{E}_w'$ and the ring $\hat{\mathcal{E}}_w'$ of entire functions $g$ of exponential type, satisfying growth conditions
\begin{equation}\label{paley-wiener}
|g|^w_{\infty, \lambda} := \sup_{\xi \in R^n} |g(\xi)| \, e^{\lambda w(\xi)} < \infty
\end{equation}
with real $\lambda = \lambda (g)$.
Moreover, $g \in \hat{\mathcal{D}}_w$ if and only if (\ref{paley-wiener}) holds for all $\lambda > 0$.
Note that a locally integrable function $U$ on $R^n_\xi$, which satisfies
(\ref{paley-wiener}), defines a distribution $u \in \mathcal{D}_w'$ by 
$\langle u, \phi\rangle := \int U(\xi) \hat{\phi} (-\xi) \, d\xi$.
Write then $u \in \mathcal{F}_w^{(\lambda)}$
and 
$\hat{u} = U \in \hat{\mathcal{F}}_w^{(\lambda)}$, 
and let $\mathcal{F}_w$ be the union of
$\mathcal{F}_w^{(\lambda)}$ over real $\lambda$ 
(observe that this is a smaller space than the one in \cite{bjorck_66} denoted by the same symbol).
One may use the semi-norms (\ref{w-norm-fourier}) also in $\mathcal{F}_w^{(\lambda)}$, that is 
$||| u |||_\lambda^w =|\hat{u}|^w_{\infty, \lambda}$, but allowing now $\lambda$ real.

\subsubsection{Convolution.}

Use the standard notation $\check{u}$ for the reflection of a distribution $u$ in the origin.
Acting by convolution, an element $\check{u} \in \mathcal{E}_w'$ gives rise to a bounded linear operator $T_{\check{u}} f (x)= (\check{u} * f) (x)= \langle \check{u}, f(x - \cdot )\rangle$ on $\mathcal{E}_w$.
The dual operator $T_{\check{u}}'$ acts on $\mathcal{E}_w'$ as convolution by $u$, and this action is transformed into multiplication by $\hat{u}$ in $\widehat{\mathcal{E}_w'}$.
The latter is a ring of entire functions with no zero-divisors, hence $T_{\check{u}}'$ is injective, and $T_{\check{u}}$ has dense range.
Thus $T_{\check{u}}$ is surjective if and only if the range of $T_{\check{u}}'$ is closed; hence if and only if the principal ideal $\hat{u} \cdot \widehat{\mathcal{E}_w'}$ is closed in $\widehat{\mathcal{E}_w'}$.
By an approximation argument, see \cite{hormander_68}, principal ideals in
$\widehat{\mathcal{E}_w'}$ are local, i.e. the closure of 
$\hat{u} \cdot \widehat{\mathcal{E}_w'}$ in $\widehat{\mathcal{E}_w'}$ 
is equal to 
$\hat{u} \cdot \mathcal{A} \cap \widehat{\mathcal{E}_w'}$, 
with $\mathcal{A}$ 
denoting the ring of all entire functions in $C^n$.
Hence the equivalence of the surjectivity of $T_{\check{u}}$ and the
equality
$\hat{u} \cdot \mathcal{A} \cap \widehat{\mathcal{E}_w'} = \hat{u} \cdot
\widehat{\mathcal{E}_w'}$, 
which, as in the Schwartz case, holds if and only if 
$\check{u}$ is $w$-\emph{slowly decreasing}: 
for some $A > 0$ and all $\xi$ with $|\xi| > 1$,
\begin{equation}\label{slowly-decreasing}
\sup_{|\eta| \leq A w(\xi)} |\hat{u}(\xi + \eta)| \geq A^{-1} \cdot e^{-A w(\xi)}.
\end{equation}
We have thus sketched the proof of a classical fact, presently our point of departure. 
\begin{theorem}\label{theorem1}
Let $u \in \mathcal{E}_w'$.
Then $u *\mathcal{E}_w = \mathcal{E}_w$ if and only if $\hat{u}$ is $w$-slowly decreasing.
\end{theorem}
The Schwartz case $w(\xi)= \log (1+ |\xi|)$ of this theorem goes back to Ehrenpreis \cite{ehrenpreis_56,ehrenpreis_60} and H\"{o}rmander \cite{hormander_62}, and extensions to other spaces are numerous, cf. e.g.
\cite{chou,cioranescu_80,cioranescu_zsido,cioranescu_zsido_2,cioranescu_zsido_arkiv}.
Distributions satisfying either of the conditions of the theorem are called, in the Ehrenpreis' tradition, $w$-invertible.
Note that $w$-invertibility is actually a property of the equivalence class of $u \in \mathcal{E}_w'$ in $\mathcal{E}_w'/\mathcal{D}_w$, since by Paley-Wiener
$|\hat{u}(\xi) + \hat{\phi}(\xi)| = |\hat{u}(\xi)| + o(e^{-\lambda w(\xi)})$
for 
$\phi \in \mathcal{D}_w$, $\lambda > 0$, and $|\xi| \rightarrow \infty$.
A distribution $u \in \mathcal{D}_w'$, not necessarily of compact support, is then naturally called $w$-invertible, if its $w$-singular support is compact and for some $\phi \in \mathcal{D}_w$ the sum $u + \phi$ is an invertible element of $\mathcal{E}_w'$; alternatively, if $\psi u \in \mathcal{E}_w'$ is $w$-invertible for some $\psi \in \mathcal{D}_w$ equal to one in a neighbourhood of the $w$-singular support of $u$.

\section{Pseudo-convolution operators} \label{operators}

Refer to \cite{hormander_pde} for pseudo-differential operators.

\subsection{Symbols} \label{symbols}

For $X \subset R^n$ open, $|\cdot|_\pi$ a semi-norm on $C^\infty(X)$, $w \in M$, and $m \in R$, define a semi-norm $|\cdot|_{m;\pi}$ on $C^\infty(X \times R^n)$ by
\begin{equation}\label{p-seminorm}
|p \,|_{m;\pi} := \sup_{\xi \in R^n} |p(\cdot, \xi)|_\pi \cdot e^{-m w(\xi)}.
\end{equation}
If $\mathcal{C}$ is a subspace of $C^\infty(X)$ constructed with semi-norms $\pi \in \Pi$,
let $\Sigma_w^m(\mathcal{C})$ be the completion of the subspace of 
$C^\infty(X \times R^n)$ 
constructed in the same way with semi-norms $|\cdot|_{m;\pi}$.
Call an element of 
$\Sigma_w^m(\mathcal{C})$ 
a symbol of order $m$ and class
$\mathcal{C}$.
Denote the union of 
$\Sigma_w^m(\mathcal{C})$ over 
$m \in R$ by $\Sigma_w(\mathcal{C})$.

In particular, recalling (\ref{wl-norm}), for $p \in \Sigma_{w_1}^m(\mathcal{E}_{w_2})$,
\begin{equation}\label{symbol-w}
|p \,|_{m; \lambda , \phi} := \sup_{\xi , \eta \in R^n} |\widehat{\phi p(\cdot, \xi)}(\eta)| \cdot e^{-m w_1(\xi) + \lambda w_2(\eta)} < \infty,
\end{equation}
for all $\lambda > 0$ and $\phi \in \mathcal{D}_{w_2}$.
In the case $w_1 = w_2 = \log (1 + |\cdot|)$, this may be equivalently written  as
\begin{equation}\label{symbol-w-schwartz}
|p \,|_{m; \alpha , K} := \sup_{\xi \in R^n, x \in K} |D^\alpha_x p(x, \xi)| \cdot (1+|x|)^{-m} < \infty,
\end{equation}
for all multi-index $\alpha \geq 0$ and $K \subset X$ compact.

Note likewise, recalling (\ref{dc-norm}), that for 
$p \in \Sigma_{w}^m(C^L)$,
for every compact $K \subset X$ there is $r>0$ such that
\begin{equation} \label{symbol-CL}
|p \,|_{m; r, K} := \sup_{\xi \in R^n, x \in K, \alpha \geq 0}
(r/L_{|\alpha|})^{|\alpha|} \cdot |D^\alpha_x p(x, \xi)| \cdot e^{-m w(\xi)} < \infty.
\end{equation}
When $C^L$ is the real analytic class, $L_k =k$, $k \geq 0$, it is easy to see that (\ref{symbol-CL}) holds if and only if the functions $p(\cdot, \xi)$, $\xi \in R^n$, have holomorphic extensions $\tilde{p}_\xi$ to a complex neighbourhood $U=U_\xi$ of $K$, and
$\sup_{z \in U, \xi \in R^n} |\tilde{p}_\xi(z)| \cdot e^{-m w(\xi)}$ is finite.

\subsection{Regular symbols of pseudo-differential operators} \label{regular-symbols}

Let as in \cite{hormander_pde}, Ch. XVII, the set $S^m$ of symbols, $m \in R$, consist of functions $p \in C^\infty(R^n \times R^n)$ such that for every $\alpha,\beta \geq 0$ there is $C_{\alpha, \beta}$ so that $|D_\xi^\alpha D_x^\beta p(x, \xi)| \leq C_{\alpha, \beta} \cdot (1 + |\xi|)^{m - |\alpha|}$ for $x, \xi \in R^n$.

Say that $p \in S^m$ is of class $C^L$ on a compact $K \subset R^n_x$ if there exist positive constants $r=r(K)$ and $R=R(K)$ such that for all $\alpha \geq 0$
\begin{equation} \label{symbol-S-CL}
\sup_{|\xi| > R} |D^\alpha_\xi p(x, \xi)|_{L,r,K} \cdot  (1 + |\xi|)^{m - |\alpha|} < \infty,
\end{equation}
with $|\cdot|_{L,r,K}$ defined by (\ref{dc-norm}).
Say that $p$ is of class $C^L$ in an open set $X$ if $p$ is of class $C^L$ on every compact in $X$.

If $C^L$ is the real analytic class, it is easy to see that (\ref{symbol-S-CL}) holds if and only if each of the functions 
$p_\xi^{(\alpha)}:= D^\alpha_\xi p(\cdot ,\xi)$,
$\alpha \geq 0$, $|\xi| > R$, 
has holomorphic extension 
$\tilde{p}_\xi^{(\alpha)}$ 
to
some complex neighbourhood $U$ of $K$, independent of $\xi$, such that
\begin{equation} \label{symbol-S-analytic}
|p|_m^{(\alpha)} = |p \,|_{m, U, R}^{(\alpha)} :=  \sup_{z \in U, |\xi| > R} |D^\alpha_\xi \tilde{p}(z, \xi)|\cdot  (1 + |\xi|)^{- m + |\alpha|} < \infty.
\end{equation}

Note that 
$S^m \subset \Sigma_w^m(\mathcal{E}_w)$ 
for $w= \log (1+|\cdot|)$, 
and then $p \in S^m$ of class $C^L$ is also of class $C^L$ as an element of
$\Sigma_w^m(\mathcal{E}_w)$.

We record the following addendum to Proposition 18.1.3 and its proof in \cite{hormander_pde}. 

\begin{proposition} \label{proposition1}
Let $p_j \in S^{m_j}$, $j= 0, 1, \ldots$, with $m_j \rightarrow - \infty$ as $j \rightarrow \infty$.
Assume $p_j$ real analytic on compact $K \subset R^n$ uniformly, that is, that for some $R > 0$ and complex neighbourhood $U$ of $K$, the holomorphic extensions of all $p_j$ satisfy (\ref{symbol-S-analytic}).
Set $m_0'= \max_{j \geq 0} m_j$.
Let $\chi \in C_0^\infty$ be equal to one in a neighbourhood of zero, let $\epsilon_j$ be a decreasing sequence of real numbers with limit zero, and put $P_j(x, \xi)= (1-\chi(\epsilon_j \xi)) \cdot p_j(x, \xi)$.
Then, provided $\epsilon_j$ approach zero rapidly enough, the symbol $p= \sum_j P_j$ belongs to $S^{m_0'}$ and it is analytic on $K$, that is, $p$ satisfies (\ref{symbol-S-analytic}) with $m=m_0'$.
\end{proposition}

\begin{proof}
Following the proof cited, we conclude that $p \in S^{m_0'}$, and, the sum $\sum_j P_j$ being locally finite, $p$ has a holomorphic extension to $U$.
It remains to show (\ref{symbol-S-analytic}) with $m=m_0'$.
Due to $|p \,|_{m_0'}^{(\alpha)} \leq \sum_j |P_j|_{m_0'}^{(\alpha)}$, $\alpha \geq 0$, the bound follows if for every index $\alpha$ there is $j(\alpha)$ such that
\begin{equation} \label{2to-j}
|P_j|_{m_0'}^{(\alpha)} < 2^j \mbox{ \hspace{.1cm} if \hspace{.1cm}} j \geq j(\alpha).
\end{equation}
But, using the Leibniz formula, for $\alpha, j \geq 0$, $z \in U$, and $|\xi| > R$, the expression
$|D^\alpha_\xi \tilde{P}_j(z, \xi)|\cdot  (1 + |\xi|)^{- m_0' + |\alpha|}$
is estimated by (a constant $C_\alpha$ times) the maximum of the functions
$(1 + |\xi|)^{m_j - m_0'} \cdot |D_\xi^\beta (1- \chi)(\epsilon_j \xi)|$,
$0 \leq \beta \leq \alpha$.
This gives 
$|P_j|_{m_0'}^{(\alpha)}  \leq C_{\alpha, j} \cdot \epsilon_j$
if $j \geq j_0$
is so large that $m_j \leq m_0' -1$.
Choosing 
$\epsilon_j < \min_{|\alpha| \leq j} 2^j/C_{\alpha, j}$, 
we get (\ref{2to-j}) with 
$j(\alpha)= \max (j_0, |\alpha|)$.
\end{proof}

\subsection{Pseudo-convolution} \label{pseudo-convolution}

A pseudo-differential operator $p(x,D)$ in $R^n$ acts in $C_0^\infty$ by the formula
$p(x,D)\phi (x) = \int e^{ixy} p(x,y) \hat{\phi} (y) \, dy$.
If $\psi \in C_0^\infty(R^n_x)$, 
\begin{equation} \label{action}
{(\psi p(x,D) \phi)}^{\hat{}} \, (\xi) = \int \widehat{\psi p(\cdot,\eta)}(\xi - \eta) \cdot \hat{\phi} (\eta) \, d\eta,
\end{equation}
which also describes the action of $p(x.D)$ on $\mathcal{E}'$.
It is easy to see that a symbol $p \in \Sigma_w^m(\mathcal{E}_w)$, used as the kernel of an integral operator on $\hat{\mathcal{F}}_w$, will likewise define an operator
$p(x,D): \mathcal{F}_w \rightarrow \mathcal{D}_w'$.

More generally, consider the set $\Omega = \Omega(w)$ of measurable functions $a$ in $R^n \times R^n$ such that for every real $\lambda$ there is $\Lambda$ such that
$$[a]_{\lambda, \Lambda} := \sup_{\xi \in R^n} e^{-\Lambda w(\xi)} \cdot \int |a(\xi, \eta)| \, e^{\lambda w(\eta)} \, d\eta < \infty.$$
Writing then
$A_a f(\xi)= \int a(\xi, \eta) f(\eta) \, d\eta$, 
it is clear that 
$|A_a f|_{\infty, \Lambda} \leq [a]_{\lambda, \Lambda} \cdot |f|_{\infty, \lambda}$, 
and hence
$A_a: \hat{\mathcal{F}}_w^{(\lambda)} \rightarrow \hat{\mathcal{F}}_w^{(\Lambda)}$ 
if
$[a]_{\lambda, \Lambda}$ 
is finite.

For 
$p \in \Sigma_w^m(\mathcal{E}_w)$
and 
$\psi \in \mathcal{D}_w$ 
let now 
$a_{\psi p}(\xi, \eta)$ denote $\widehat{\psi p(\cdot, \eta)} (\xi - \eta)$.
\begin{lemma} \label{lemma1}
If $p \in \Sigma_w^m(\mathcal{E}_w)$ and $\psi, \psi_1 \in \mathcal{D}_w$, and $\lambda,
\Lambda$ are real numbers, then
\begin{enumerate}
  \item $[a_{\psi p}]_{\lambda, m+\lambda} \leq |p \,|_{m;\lambda, \Lambda} \cdot \int e^{(|m + \lambda| -\Lambda) w(\eta)} \, d\eta$
  \item $[a_{\psi_1 \psi p}]_{\lambda, \Lambda} \leq \|\psi_1\|_{|\Lambda|} \cdot [a_{\psi p}]_{\lambda,
  \Lambda}$.
\end{enumerate}
\end{lemma}
\begin{proof}
The first estimate follows by (\ref{symbol-w}).
The second follows by direct computation, observing that
$a_{\psi_1 \psi p} (\xi, \eta) = \int \hat{\psi_1}(\xi - \eta') a_{\psi p}(\eta', \eta)
\, d\eta'$. 
\end{proof} 
It thus follows that for any real $\lambda$ the operator $p(x,D)$ maps
$\mathcal{F}_w^{(\lambda)}$ to $(\mathcal{F}_w^{(\lambda + m)})_{loc}$.
In particular, smooth functions in $\mathcal{F}_w$ are mapped to smooth functions.

\subsection{Elliptic analytic pseudo-differential operators} \label{pseudo-differential}

A symbol $p \in S^m$ and the associated operator $p(x,D)$ are said to be elliptic in a compact set $K \subset R^n$ if for some constants $c > 0$, $C \geq 0$, and some open neighbourhood $X$ of $K$
\begin{equation}\label{elliptic}
|p(x,\xi)| \geq c \cdot |\xi|^m  \mbox{ for } x \in X \mbox{ and } |\xi| > C.
\end{equation}
The operator $p(x,D)$ is said to be analytic on $K$ if $p$ is analytic on $K$ as a symbol.
Recall a basis fact.
\begin{theorem}\label{elliptic-inverse}
If $p \in S^m$ is elliptic and analytic on a compact $K$, then there is $q \in  S^{-m}$
analytic on $K$ such that $q(x,D)$ is the left inverse of $p(x,D)$ on Schwartz
distributions $\mathcal{S}'$ of tempered growth modulo smoothing operators.
\end{theorem}
The proof of Theorem 18.1.9 in \cite{hormander_pde} applies, but noting that the parametrix
there constructed is now analytic on $K$.
Indeed, the step $(iv) \Rightarrow (iii)$ there preserves analyticity on $K$, involving
composition with $F(z)=1/z$ holomorphic in $|z|> c$, and the step $(iii) \Rightarrow
(ii)$ preserves analyticity by Proposition \ref{proposition1}.

\section{Coercing invertibility} \label{coertion}

Write generically $\Phi$ for a space of smooth functions in $R^n$ with dual $\Phi'$ consisting of (ultra-) distributions with compact support.
Following Ehrenpreis \cite{ehrenpreis_56,ehrenpreis_60}, one says that $u \in \Phi'$ is $\Phi$-invertible if $u$ acts surjectively in $\Phi$ via convolution.
Let $\pi: \Phi' \rightarrow \Phi'$ and let $\Phi' \subset \Phi_1'$.
We say that $\pi$ is $(\Phi, \Phi_1)$-coercive if $u \in \Phi'$ and $\pi u$ $\Phi$-invertible implies $u$ $\Phi_1$-invertible; write $\Phi$-coercive for $(\Phi, \Phi)$-coercive.
(But we avoid such formal articulation, when the spaces are clear from the context.)

\subsection{Sufficient conditions} \label{sufficient-conditions}

The coercivity of 
$p(x,D)$, $p \in \Sigma_{w}^m(\mathcal{C})$, 
follows, cf. Theorem \ref{theorem1}, from estimates of the Fourier transform of $p(x,D)u$ by that of $u$.
Recall from Section \ref{pseudo-convolution} the definition of the kernel set $\Omega$.
\begin{lemma}\label{lemma2}
Let $w \prec w'$ be slowly varying functions in $M$.
Let 
$\mathcal{G} \subset \Omega=\Omega(w)$ 
be bounded in the sense that for every $\lambda$ there is $\Lambda$ such that 
$\sup_{a \in \mathcal{G}} [a]_{\lambda, \Lambda}$
is finite.
Suppose that for some real $\lambda_0$ and all $\lambda > 0$ there is $\rho > 0$ such
that
\begin{equation}\label{lemma2-estimate}
\inf_{a \in \mathcal{G}} \int_{|\eta - \xi| > \rho w'(\xi)} |a(\xi, \eta)| \cdot
e^{\lambda_0 w(\eta)} \, d\eta = o(e^{-\lambda w(\xi)})
\end{equation}
as $|\xi| \rightarrow \infty$.
Then (still writing 
$A_a f(\xi)$ for $\int a(\xi, \eta) \, f(\eta) \, d\eta$), if $f \in \hat{\mathcal{F}}_w$ and the function $\inf_{a \in \mathcal{G}} |A_af|$ is $w$-slowly decreasing, then $f$ is $w'$-slowly decreasing.
\end{lemma}
\begin{proof}
For $a \in \Omega$, $\rho > 0$, and $\lambda', \Lambda$ real, estimate the integral
$A_af(\xi)$ by the sum of the integrals of $|a(\xi, \eta) f(\eta)|$ over two regions,
$E_1 := \{\eta: |\eta - \xi| \leq \rho w'(\xi)\}$ and its complement $E_2$.
Bound the first integral by
$$[a]_{0, \Lambda} \cdot e^{\Lambda w(\xi)} \cdot \sup_{|\eta|
\leq \rho w'(\xi)} |f(\xi + \eta)|.$$
Bound the second integral by
$$[a]_{\infty, -\lambda'} \cdot \int_{E_2} |a(\xi, \eta)| \cdot e^{\lambda' w(\eta)} \, d\eta,$$
and use the Cauchy-Schwartz inequality to bound the square of the integral in the last term by the product of
$$\int_{R^n} |a(\xi, \eta)| \cdot e^{(2 \lambda' - \lambda_0) w(\eta)} \, d\eta$$
and
$$\int_{|\eta - \xi| > \rho w'(\xi)}|a(\xi, \eta)| \cdot e^{\lambda_0 w(\eta)} \,
d\eta.$$
The second last integral in bounded by 
$[a]_{2\lambda' -\lambda_0, \Lambda'} \cdot e^{\Lambda' w(\eta)}$
for any $\Lambda'$.
Choose now $\lambda'$ so that $[a]_{\infty, -\lambda'}$ is finite, and choose $\Lambda, \Lambda'$ so that $[a]_{0, \Lambda}$ and $[a]_{2 \lambda' - \lambda_0, \Lambda'}$ are bounded if $a \in \mathcal{G}$.
Putting it all together, we obtain
\begin{equation}\label{main-estimate}
\inf_{a \in \mathcal{G}} |A_af(\xi)| \leq C \, e^{C w(\xi)} \, \left(\sup_{|\eta| \leq \rho w'(\xi)} |f(\xi + \eta)| + o(e^{-\lambda w(\xi)})\right) \mbox{\hspace{.4cm}} (\forall \lambda \exists \rho),
\end{equation}
with a constant $C$ depending on $f$ and $\mathcal{G}$ only.
Take the maximum of both sides of (\ref{main-estimate}) over a ball $B(\xi, A w(\xi))$,
$A > 0$, recalling that the left side is $w$-slowly decreasing, and that $w$ and $w'$ are
slowly varying; we get
$$C_1^{-1} e^{C_1 w(\xi)} \leq \sup_{|\eta| \leq C_1 w(\xi) + \rho C_1
w'(\xi)} |f(\xi + \eta)| + o(e^{-\lambda w(\xi)})  \mbox{\hspace{.3cm}} (\exists C_1
\forall \lambda \exists \rho).$$
Since $w' \succ w$, this obviously implies $w'$-slow decrease of $f$.
\end{proof} 

Before stating the main result, introduce for brevity of statement the following conditions; notice that each of these conditions implies that $w' \succ w$.

\begin{description}
  \item[(*)] Given $w,w' \in M$, $p \in \Sigma_w(C^L)$, suppose that for all $b>0$  there exist $a,R >0$, such that
\begin{equation}\label{*}
  q_L(aw'(\xi)) \geq b w(\xi) \ \mbox{ if } \ |\xi| > R.
\end{equation}
  \item[(**)] Given $w,w' \in M$, $p \in \Sigma_w(\mathcal{E}_\gamma)$ with $\gamma \in M$ of form $\gamma(\xi)=\Gamma(|\xi|)$ and $\Gamma$ non-decreasing, suppose 
\begin{equation}\label{**}
  \Gamma \circ w' \succ w.
\end{equation}
\end{description}

\begin{theorem} \label{thm-main}
Let $w, w' \in M$ be slowly varying, $u \in \mathcal{F}_w$, and $p \in
\Sigma_w^m(\mathcal{E}_w)$.
Let $U$ be an open neighbourhood of the $w$-singular support of $p(x,D)u$, assumed
compact, with $\bar{U}$ compact.
Assume $p(x,D)u$ $w$-invertible, that is, that the Fourier transform of $\psi p(x,D) u$ is $w$-slowly decreasing when $\psi \in \mathcal{D}_w$ is equal to one in $U$.
Then either of the conditions (*), (**), implies that $\hat{u}$ is $w'$-slowly decreasing; hence, if the $w'$-singular support of $u$ is compact, then $u$ is $w'$-invertible.
\end{theorem}

\noindent \emph{Remark.}  When $C^L$ is the real analytic class, the condition (*) is
satisfied if $w' \succ w$.
The conclusion of Theorem \ref{thm-main} thus holds with $w' \sim w$ in this case.

\begin{proof}
The condition (*) implies $w'$-slow decrease of $\hat{u}$ by Lemma \ref{lemma2} with 
$$\mathcal{G} = \{a_{\chi_N p}: N=0, 1, \ldots\},$$ 
where the functions 
$\chi_N \in \mathcal{D}_w(U')$, $U' \subset\subset U$, 
are all equal to one in a fixed neighbourhood of the $w$-singular support of $p(x,D)u$, and satisfy
\begin{equation}\label{ehrenpreis-unit}
|D^\alpha \chi_N| \leq (C N)^{|\alpha|} \mbox{ if } |\alpha| \leq N, \mbox{\hspace{.2cm}}
N= 0, 1, \ldots .
\end{equation}
The existence of such functions is a standard fact, see e.g. Theorem 1.4.2 in
\cite{hormander_pde}.
Recall that they may take the form 
$\chi_N= \Phi * \phi_{(N)} * \ldots * \phi_{(N)}$ 
($N$-fold convolution), where $\Phi \in \mathcal{D}_w(U')$ is one in a neighbourhood of the $w$-singular support of $p(x,D)u$, the function $\phi \in C_0^\infty$ is non-negative with integral one and support sufficiently close to zero, and $\phi_{(N)}(x)=N^n \phi(Nx)$.

We verify that the assumptions of Lemma \ref{lemma2} are satisfied, in three steps.

\vspace{.2cm}

\emph{STEP 1: $\mathcal{G}$ is bounded.}
By Lemma \ref{lemma1},
$$[a_{\chi_N p}]_{\lambda, \Lambda} = [a_{\chi_N \psi p}]_{\lambda, \Lambda} \leq \|\chi_N\|^w_\lambda \cdot [a_{\psi p}]_{\lambda, \Lambda},$$
and
$$\|\chi_N\|^w_\lambda = \int |\hat{\Phi}| \cdot |\widehat{\phi_{(N)}}|^N \cdot e^{\lambda w(\xi)} \, d\xi \leq \|\Phi\|^w_\lambda \cdot \| \widehat{\phi_{(N)}} \|_{L^\infty}^N \leq \|\Phi\|^w_\lambda
\cdot \| \phi_{(N)} \|_{L^1}^N = \|\Phi\|^w_\lambda.$$
Thus any 
$[\cdot]_{\lambda, \Lambda}$ with $[a_{\psi p}]_{\lambda, \Lambda}$ 
finite is bounded on $\mathcal{G}$, which by Lemma \ref{lemma1} happens if 
$\Lambda \geq \lambda + m$.

\vspace{.2cm}

\emph{STEP 2: The function $\inf_N |A_N \hat{u}|$, with $A_N$ given by the kernel $a_{\chi_N p}$, is $w$-slowly decreasing.}
For this, we need a lemma.
\begin{lemma} \label{sublemma}
Let 
$v \in \mathcal{F}_w$, $\chi, \phi \in \mathcal{D}_w$,
and
$\phi \chi = \phi$.
Then for any real $\lambda$ and $\xi \in R^n$,
\begin{equation} \label{sublemma-estimate}
|\widehat{\chi v}(\xi)| \geq |\hat{v}(\xi)| - |||(1-\phi)v |||_\lambda^w \cdot (1+ \|\chi\|_\lambda^w) \cdot e^{-\lambda w(\xi)}.
\end{equation}
\end{lemma}

\begin{proof}
Write
$$v = \chi v + (1- \chi)v = \chi v + (1- \chi)(1- \phi)v.$$
The Fourier transform of the last term times $e^{\lambda w(\xi)}$ is clearly bounded by
$$|||(1- \chi)(1- \phi)v |||_\lambda^w,$$
recalling (\ref{w-norm-fourier}), in turn bounded by
$$|||(1- \chi)v |||_\lambda^w + |||(1- \phi)v |||_\lambda^w \leq |||(1-\phi)v |||_\lambda^w \cdot (1+ \|\chi\|_\lambda^w).$$
\end{proof}

Now $A_N \hat{u}$ is the Fourier transform of $\chi_N \cdot (\psi p(x,D)u)$.
Use Lemma \ref{sublemma} with $\chi=\chi_N$, $v= \psi p(x,D)u$, and $\phi \in \mathcal{D}_w$ equal to one near the $w$-singular support of $v$ and with support in the interior of the set where $\chi_N=1$.
But 
$(1-\phi)v \in \mathcal{F}_w \cap \mathcal{E}_w$, 
hence 
$|||(1- \chi)v |||_\lambda^w < \infty$ for any $\lambda \in R$, 
and by STEP 1, 
$\|\chi_N\|_\lambda^w \leq \|  \Phi \|_\lambda^w$ 
for all $N \geq 0$.
It follows that for $\lambda > 0$ the infimum of $|A_N \hat{u} (\xi)|$ over $N \geq 0$ is bounded from below by the modulus of the Fourier transform of $\psi p(x,D)u$ at $\xi$ plus $o(e^{-\lambda w(\xi)})$.
This concludes STEP 2.

\vspace{.2cm}

\emph{STEP 3: The estimate (\ref{lemma2-estimate}).}
Let $\lambda_0$ be such that the integral $\int e^{(m-\lambda_0) w(\xi)} \, d\xi$ is finite, and show that for all $\lambda > 0$ there is $\rho > 0$ such that
\begin{equation} \label{step3-estimate1}
\inf_{N \geq 0} \ \int_{|\eta - \xi| > \rho w'(\xi)} |\widehat{\chi_N p(\cdot, \eta)}(\xi - \eta)| \cdot e^{-\lambda_0 w(\eta)} \, d\eta = o(e^{-\lambda w(\xi)})
\end{equation}
as $|\xi| \rightarrow \infty$.
Begin by estimating the integrand in a standard way.
Bound for $\alpha \geq 0$ the product of $|\xi^\alpha|$ and the modulus of the Fourier transform of 
$\chi_N p(\cdot, \eta)$ at $\xi$ 
by the $L_1$ norm of $D_x^\alpha (\chi_N p(\cdot,\eta))$, 
in turn is bounded by its $L_\infty$ norm times the volume of $U$.
Apply Leibniz formula, noting (\ref{ehrenpreis-unit}), and that
$$\sup_{x \in \bar{U}'} |D_x^\beta p(x, \xi)| \leq |p \,|_{m; r, \bar{U}'} \cdot (L_{|\beta|}/r)^{|\beta|} \cdot e^{m w(\eta)}$$
for some $r > 0$ and all $\beta, \eta$, the symbol $p$ being of class $C^L$ on
$\bar{U}'$; in this way, bound the Fourier transform of $\chi_N p(\cdot, \eta)$ at $\xi$ by
$C (L_N/r |\xi|)^N \cdot e^{m w( \xi)}$, $N= 0, 1, \ldots$.
The integral in (\ref{step3-estimate1}), written as as the integral of 
$|\widehat{\chi_N p(\cdot, \xi - \eta)}(\eta)| \cdot e^{- \lambda_0 w(\xi - \eta)}$
over the set 
$\{ |\eta| > \rho w'(\xi) \}$, 
is then bounded by
$$ C (L_N/r)^N \cdot e^{(\lambda_0 - m) \cdot w( \xi)} \cdot \int_{|\eta| > \rho w'(\xi)} |\eta|^{-N} \cdot e^{(\lambda_0 - m) \cdot w( \xi)} \, d\eta $$
and thus by
$$ e^{(\lambda_0 - m) \cdot w( \xi)} \cdot \int_ {R^n} e^{(m - \lambda_0) w(\eta)} \, d\eta \cdot C (L_N/(r\rho w'(\xi)))^N,$$
with the integral in the last expression finite by assumption.
Taking the infimum over $N \geq 0$, gives a bound for the left side in (\ref{step3-estimate1}) of form
$$C \cdot e^{(\lambda_0 - m) \cdot w( \xi) - q_L(r\rho w'(\xi))},$$
which is $o(e^{-\lambda w(\xi)})$ as $|\xi| \rightarrow \infty$ provided $\rho > a/r$,
where $a$ is a constant of (*) corresponding to some 
$b > \lambda + \lambda_0 -m$.
This completes STEP 3, and thus the proof of the (*) case.

For case (**) use Lemma \ref{lemma2} with $\mathcal{G} = \{a_{\psi p}\}$. STEP 1 and 2 are here trivial; for STEP 3, estimate 
$\widehat{\psi p(\cdot, \eta)}(\xi)$ by $|p|_{m;\lambda, \psi} \cdot
e^{m w(\eta) - \lambda \gamma(\xi)}$, 
and then bound the integral
$$\int_{|\eta| > w'(\xi)} |\widehat{\psi p(\cdot, \xi - \eta)}(\eta)| \cdot
e^{-\lambda_0 w(\xi - \eta)}$$
by
$$|p|_{m;\lambda, \psi} \cdot \int_{|\eta| > w'(\xi)} e^{(m- \lambda_0) w(\xi - \eta) - \lambda \gamma(\eta)} \, d\eta,$$
which, by the monotonicity of $\Gamma$, is bounded by
$$|p|_{m;\lambda, \psi} \cdot \int_{R^n} e^{(m- \lambda_0) w(\eta)} \, d\eta \cdot e^{(\lambda_0 - m) w(\xi) -\lambda \Gamma (w'(\xi))}.$$
Clearly, condition (**) makes the last expression 
$o(e^{- \Lambda w(\xi)})$ as $\xi \rightarrow \infty$ 
for any $\Lambda > 0$ 
with $\lambda=\lambda(\Lambda)$ large enough.

The proof of the theorem is now complete.
\end{proof} 

Record explicitly some consequences of Theorem \ref{thm-main}.

\begin{theorem} \label{analytic-multiplier}
Let $w \in M$ be slowly varying, $u \in \mathcal{E}_w'$, and let $f \in \mathcal{E}_w)$ be real analytic in a neighbourhood of the $w$-singular support of $u$.
Then, if $fu$ is $w$-invertible, then so is $u$.
\end{theorem}

\begin{proof}
By Remark to Theorem \ref{thm-main} and the fact that the $w$-singular supports of $u$ and $fu$ coincide.
\end{proof} 

\begin{theorem} \label{thm-w-multiplier}
Let $w, w' \in M$ be slowly varying, $u \in \mathcal{E}_w'$, and let $f \in \mathcal{E}_w$ be of class $\mathcal{E}_\gamma$ in a neighbourhood of the $w$-singular support of $u$, where $\gamma \in M$ has the form $\gamma(\xi)= \Gamma(|\xi|)$ with $\Gamma$ non-decreasing and $\Gamma \circ w' \succ w$.
Then, if $fu$ is $w$-invertible, then $u$ is $w'$-invertible.
\end{theorem}

\begin{proof}
Obvious by the condition (**).
\end{proof} 

Specialize now to Gevrey classes, recalling that 
$\mathcal{E}_{(\alpha)}$ 
denotes
$\mathcal{E}_{w_{\alpha}}$,  
$w_{\alpha}(\xi)= |\xi|^\alpha$, $0 < \alpha < 1$, 
and that
$\mathcal{E}_{(1)}$ denotes the real analytic class.

\begin{theorem} \label{gevrey-multiplier}
Let $0 < a \leq 1$ and $0 < r,s < 1$ satisfy $as \geq r$.
Let $u \in \mathcal{E}_{(r)}'$, and $f \in \mathcal{E}_{(r)}$ be of class $\mathcal{E}_{(a)}$ in a neighbourhood of the $\mathcal{E}_{(r)}$-singular support of
$u$.
Then, if $fu$ is $\mathcal{E}_{(r)}$-invertible, then $u$ is
$\mathcal{E}_{(s)}$-invertible.
\end{theorem}

\begin{proof}
Follows by Theorem \ref{analytic-multiplier} if $a=1$, and by Theorem \ref{thm-w-multiplier} if $0 < a < 1$.
\end{proof} 

\begin{theorem} \label{thm-pdo}
Let $u \in \mathcal{E}'$ and let $p(x,D)$ be a pseudo-differential operator in $R^n$ elliptic and analytic on the singular support of $u$.
Let $\psi \in C^\infty_0(R^n)$ be one in a neighbourhood of the singular support of $u$.
Then $\psi p(x,D) u$ is invertible if and only if $u$ is invertible.
\end{theorem}

\begin{proof}
That 
``$\psi p(x,D) u$ invertible implies $u$ invertible''
follows by Theorem \ref{thm-main} and the Remark there, for by the pseudo-local property of pseudo-differential operators, $p(x,D) u$ is analytic on the singular support of $p(x,D)u$.
Conversely, by Theorem~\ref{elliptic-inverse} and the pseudo-local property,
$$ \psi q(x,D) (\psi p(x, D) u) - u \in C_0^\infty$$
for some $q(x,D)$ analytic on the singular support of $\psi p(x,D)u$.
Hence if $u$ is invertible, then so is $\psi p(x,D)u$, by the first part of the theorem.
\end{proof}

\subsection{Analyticity is necessary} \label{counter-examples}

The analyticity of multiplier in Theorems \ref{analytic-multiplier} and \ref{gevrey-multiplier} cannot be relaxed, as the following construction shows (suggested to the author by Lars H\"{o}rmander in a conversation at the Mittag-Leffler Institute back in 1986.)

Let, from now on, the dimension of the underlying Euclidean space equal one. 
Start by recording two elementary technical facts. 
\begin{lemma}  \label{lemma-pre-last}
Let $\nu$ be a symmetric positive measure on the real line, and let $E$ be the union of a family of open intervals, each of length at least two.
Let $r(\xi)$ denote the distance from $\xi \in R$ to the set $E$.
Then,
\begin{equation} \label{eq-lemma-pre-last}
\nu( (r(\xi), r(\xi) + 1) ) \leq  \nu * \chi_E (\xi) \leq 2 \nu((r(\xi), + \infty)).
\end{equation}
\end{lemma}
\begin{proof}
By translation and reflection invariance, we may assume $\xi=0$.
Let $\eta_0$ be a point in the closure of $E$ nearest the origin.
If $\eta_0 > 0$, then $(\eta_0, \eta_0 +1 ) = (r(0), r(0) + 1) \subset E$; if $\eta_0 < 0$, then $(\eta_0 - 1, \eta_0) = - (r(0), r(0) + 1) \subset E$; and, at least one of these inclusions holds if $\eta_0 = 0$.
Hence, by the symmetry of $\nu$, the left estimate in (\ref{eq-lemma-pre-last}) follows.
The right estimate follows similarly from the inclusion
$E \subset (-\infty, -r(0)) \cup (r(0), \infty)$.
\end{proof} 
\begin{lemma}  \label{lemma-last}
Let $f,g \in L^1(R)$ with $\hat{f},\hat{g} \in L^1(R)$ non-negative and $\hat{f}$ non-increasing in $R_+$.
Let 
$I_j = [\xi_j - d_j, \xi_j + d_j] \subset R_+$, $j= 1, 2, \ldots$, $\xi_j \rightarrow \infty$, $d_j = o(\xi_j)$, 
be a sequence of intervals distant from each other by at least two.
Put $E = R \setminus \bigcup_j I_j$, write $r(\xi)$ for the distance of $\xi \in R$ to the set $E$, and write $\chi_E$ for the characteristic function of $E$.
Define finally a pseudo-measure $u$ by $\hat{u} = \hat{g} * \chi_E$.
Then,
\begin{equation}\label{eq1-lemma-last}
\hat{u}(\xi) \, \leq  \, 2 \int_{d_j/2}^\infty \hat{g}(t) \, dt \ \mbox{ if }  \ |\xi - \xi_j| \leq d_j/2, \mbox{ \hspace{0.2cm} } j= 1, 2, \ldots,
\end{equation}
and,
\begin{equation}\label{eq2-lemma-last}
\widehat{fu}(\xi) \geq  \int_{-1}^0 \hat{g}(t) \, dt \ \cdot \hat{f}(r(\xi) + 2) \, \mbox{ for all } \xi \in R.
\end{equation}
\end{lemma}

\begin{proof}
Direct application of Lemma \ref{lemma-pre-last} gives
\begin{equation} \label{eq-something}
\int_{r(\xi)}^{r(\xi)+1} \hat{g} (t) \, dt \leq \hat{g} * \chi_E (\xi) \leq 2 \cdot \int_{r(\xi)}^{+\infty} \hat{g} (t) \, dt
\end{equation}
for all $\xi \in R$.
If $|\xi - \xi_j| \leq d_j$ for some $j$, then $r(\xi) \geq d_j/2$, yielding (\ref{eq1-lemma-last}) by the right inequality in (\ref{eq-something}).
To see (\ref{eq2-lemma-last}), use the left inequality in (\ref{eq-something}) on $\widehat{fu} = \widehat{fg} * \chi_E$, observing that
$$\widehat{fg}(t) \geq \int_{-1}^0 \hat{f} (t - s) \hat{g} (s) \, ds \geq \int_{-1}^0 \hat{g} (s) \, ds \ \cdot \hat{f}(t + 1), \mbox{\hspace{.2cm}} t \geq 0,$$
and the integral of $\hat{f}(t+1)$ over $r(\xi) \leq t \leq r(\xi) + 1$ is bounded from below by $\hat{f}(r(\xi)+2)$.
\end{proof} 

Return to the mainstream.
Write for brevity $w \in \tilde{M}$ if $w \in M$ and for some $\xi_j \rightarrow \infty$ and any $\rho_j=o(\xi_j)$ there is a constant $c > 0$ such that
\begin{equation} \label{M-tilde}
\min_{|\eta| \leq \rho_j} w(\xi_j + \eta) \geq w(\xi_j), \mbox{ \hspace{.3cm}} j = 1, 2, \ldots .
\end{equation}
Any slowly varying function in $M$, note, in particular one non-decreasing in the radius, belongs to $\tilde{M}$. 

\begin{theorem} \label{thm-counter1}
Suppose $w \in \tilde{M}$ and $\phi \in M$.
Then there exist a function $f \in \mathcal{E}_\phi$ and a pseudo-measure $u$ with compact support, such that $fu$ is $w$-invertible but $u$ is not $w$-invertible.
\end{theorem}

\begin{proof}
Let $\tilde{\phi} \succ\succ \phi + w$ be a strictly increasing function in $M$; this is the case if $\tilde{\phi}$ is a concave strict majorant of $\phi + w$, see \cite{bjorck_66} Theorem 1.2.7 and the Remark to that theorem.
Take $\xi_j \rightarrow \infty$ as in (\ref{M-tilde}) and let $d_j$ be the solutions to the equations
\begin{equation} \label{dj-equations}
\tilde{\phi} (d_j) = \min_{|\eta| \leq d_j} w(\xi_j + \eta), \mbox{ \hspace{.3cm}} j=1, 2, \ldots;
\end{equation}
these exist (uniquely) by continuity (and monotonicity) of the functions involved.
Note that $d_j = o(\xi_j)$; otherwise, after passing to a subsequence, we would have $d_j \geq \xi_j/N$ for some $N > 0$ and all $j \geq 1$, implying
$$N^{-1} \tilde{\phi} (\xi_j) \leq \tilde{\phi} (\xi_j/N) \leq \tilde{\phi} (d_j) \leq w(\xi_j),  \mbox{ \hspace{.3cm}} j=1,2, \ldots,$$
which contradicts $\tilde{\phi} \succ\succ \phi + w$.
Hence (\ref{M-tilde}) holds with $\rho_j=d_j$, and
\begin{equation} \label{dj-bound}
\tilde{\phi} (d_j) \geq c \cdot w(\xi_j), \mbox{ \hspace{.3cm}} j=1, 2, \ldots.
\end{equation}
Let now $\gamma \succ\succ \tilde{\phi}$ be a strictly increasing function in $M$.
Writing $t_j= \tilde{\phi} (s_j)$ for the right side in (\ref{dj-bound}), it follows by the monotonicity of $\gamma$ and $\tilde{\phi}$ that
$$\frac{\gamma(d_j)}{c w(\xi_j)} \geq \frac{\gamma(s_j)}{\tilde{\phi}(s_j)},
\mbox{ \hspace{.3cm}} j=1, 2, \ldots,$$
and hence, since $\gamma \succ\succ \tilde{\phi}$ and $s_j \rightarrow \infty$, we have
\begin{equation} \label{dj-bound-2}
w(\xi_j) = o(\gamma(d_j)) \mbox{ as } j \rightarrow \infty.
\end{equation}
Apply now Lemma \ref{lemma-last} with $\xi_j$ and $d_j$ as above, $\hat{f} =
e^{-\tilde{\phi}}$, and $0 \neq g \in \mathcal{D}_\gamma$, $\hat{g} \geq 0$.
Then, (\ref{eq1-lemma-last}), and the estimate
\begin{equation} \label{eq-another}
\int_r^\infty \hat{g}(t) \, dt \leq \| g \|_\lambda^\gamma \cdot e^{-\lambda \gamma(r)},
\mbox{ \hspace{.2cm} } \lambda,r > 0,
\end{equation}
valid since $\gamma$ is increasing, together give
$\hat{u}(\xi) \leq \| g \|_\lambda^\gamma  e^{-\lambda \gamma(d_j/2)}$
if $|\xi - \xi_j| \leq d_j/2$.
Note that $\gamma(d_j/2) \geq \gamma(d_j)/2$, and use (\ref{dj-bound-2}) with the implied fact that $w(\xi_j)=o(d_j)$, to conclude that $\hat{u}$ is not $w$-slowly decreasing.

On the other hand, (\ref{eq2-lemma-last}) with $\hat{f}= e^{- \tilde{\phi}}$, and the sub-additivity of $\tilde{\phi}$, give
\begin{equation} \label{eq-another-yet}
\widehat{fu}(\xi) \geq \int_{-1}^0 \hat{g}(t) \, dt \ \cdot e^{-\tilde{\phi}(2)} \cdot
e^{-\tilde{\phi}(r(\xi))}
\end{equation}
for all $\xi \in R$.
In the notation of Lemma \ref{lemma-last}, if $\xi \in E$ then $r(\xi)=0$, and if $\xi \notin E$ then $|\xi - \xi_j| \leq d_j$ for some $j$, hence $r(\xi) \leq d_j$ and
$\tilde{\phi}(r(\xi)) \leq \tilde{\phi}(d_j) \leq w(\xi)$ by the defining property of
$d_j$.
Hence (\ref{eq-another-yet}) holds with $\tilde{\phi}(r(\xi))$ replaced by $w(\xi)$, showing that $\widehat{fu}$ is $w$-slowly decreasing.
\end{proof} 

\begin{theorem} \label{thm-counter2}
Let $0 < r,s < 1$ and $0 < a \leq 1$.
If $a < r/s$ then there is $f \in \mathcal{E}_{(a)}$ and a pseudo-measure $u$ with
compact support such that $fu$ is $\mathcal{E}_{(r)}$-invertible but $u$ is not
$\mathcal{E}_{(s)}$-invertible.
\end{theorem}
\begin{proof}
Choose $\alpha$ in the interval $(\max (a,r), r/s)$, clearly non-void, and then choose
$\beta \in ((\alpha s)/r, 1)$.
Let $d_j= \xi_j^{r/\alpha}$, with $\xi_j \rightarrow \infty$ so chosen that the assumptions of Lemma \ref{lemma-last} are satisfied (note that $r < \alpha$).
Define $f \in \mathcal{E}_{(a)}$ by $\hat{f}(\xi)= e^{-|\xi|^\alpha}$ and let $g \in D_{(\beta)}$ be non-trivial with $\hat{g} \geq 0$.
Then (\ref{eq1-lemma-last}) and (\ref{eq-another}) imply
$$\hat{u}(\xi) \leq 2 \|g\|_1^{(\beta)} \cdot e^{-( \frac{1}{2}\xi_j^{r/\alpha})^\beta}  \ \mbox{ if }  \ |\xi - \xi_j| \leq \frac{1}{2} \xi_j^{r/\alpha},  \mbox{\hspace{.2cm}} j = 1, 2, \ldots.$$
By our choice of $\alpha, \beta$, we have
$$\frac{\xi_j^{r/a}}{\xi_j} \geq \frac{\xi_j^{(r/\alpha)\beta}}{\xi_j^s} =
\xi_j^{\frac{r\beta}{\alpha} -s} \rightarrow \infty \mbox{ as } j \rightarrow \infty,$$
showing that $u$ is not $\mathcal{E}_{(s)}$-invertible.

On the other hand, by (\ref{eq1-lemma-last}), repeating the last part of the proof of Theorem \ref{thm-counter1} with $\tilde{\phi} (\xi)= |\xi|^\alpha$ and observing that
$\tilde{\phi} (d_j)= (\xi_j^{r/\alpha})^\alpha \leq (2\xi)^r$
if $|\xi - \xi_j| \leq \xi_j/2$, 
we get for large $|\xi|$
$$\widehat{fu}(\xi) \geq \int_{-1}^0 \hat{g} (t) \, dt \ \cdot e^{-2^\alpha} \cdot e^{-2^r \cdot |\xi|^r}.$$
Thus $fu$ is $\mathcal{E}_{(r)}$-invertible.
\end{proof} 

\section{Two remarks}

The present is an elaboration, and tedious at that, of a direct estimate from below of certain integrals, much as in \cite{abramczuk}. 
Any advantages of directness notwithstanding, it would appear natural to quickly extend everything towards micro-local analysis,  in the spirit of H\"{o}rmander's set $\mathcal{H}(u)$ of supporting functions generalising the singular support of $u$. 
Such ideas are quite explicit in Chapter XVI of the treatise \cite{hormander_pde}, but appear still not followed up.

The original motivation for this work was to tell when some naturally arising distributions, such as measures supported by thin sets, or even indicator functions of sets, were invertible. 
Neither this direction appears to have been followed up.

\end{document}